\documentclass[11pt,final]{article}
\usepackage{amsmath,latexsym,amssymb,amsmath,
amscd,amsthm,amsxtra}
\newenvironment{definition}[1][Definition]
{\begin{trivlist}
\item[\hskip \labelsep {\bfseries #1}]}{\end{trivlist}}
\newtheorem{theorem}{Theorem}[section]
\newtheorem{proposition}{Proposition}[section]
\newtheorem{lemma}{Lemma}[section]

\def\rz{\mathbb{R}}

\def\O{\Omega}
\def\p{\partial}

\def\s{\sigma}

\def\<{\langle}
\def\>{\rangle}

\begin{document}

\title{The Neumann Problem for Hessian Equations}
\author{Xinan Ma \footnote{Department of Mathematics, University of Science and Technology of China,  Hefei , P. R. China \quad Email: xinan@ustc.edu.cn. }    \and Guohuan Qiu \footnote{Department of Mathematics, University of Science and Technology of China,  Hefei , P. R. China \quad Email: guohuan@mail.ustc.edu.cn.}}

\date{}
\maketitle

\abstract{}
In this paper, we prove the existence of a classical solution to a Neumann boundary problem for Hessian equations in uniformly convex domain. The methods depend upon the established of a priori derivative estimates up to second order. So we give a affirmative answer to a conjecture of N. Trudinger in 1986.


\section{Introduction}
  In this paper, we study the existence of the classical solution for the following  Neumann problem:
\begin{equation}\label{EMQX}
\left\{
\begin{array}{rccl}
\s_k (D^2 u) &=& f(x,u) & \hbox{ in } \Omega,\\
u_\nu&=&\varphi(x,u)  &\hbox{ on } \partial \Omega,\\
\end{array}
\right.
\end{equation}
where $\s_k(D^2 u)$ is the $k$-th elementary symmetric function of eigenvalues of $D^2 u$, and $\nu$ is outer unit normal vector of $\partial \Omega$.
When $k=1$, this is well-known Laplace equation with Neumann boundary condition, for priori estimates and the existence theorem we refer to the book \cite{GT}. For $k=n$, the priori estimates and existence result were obtained by Lions, Trudinger and Urbas \cite{LTU}. But for $2\leq k \leq n-1$, Trudinger \cite{Tru} established the existence theorem when the domain is a ball, and he conjectured (in \cite{Tru}, page 305) that  one can solve the problem  in sufficiently smooth uniformly convex domains.
Now we give a positive answer to this problem.
\begin{theorem}\label{main}
  Let $\Omega$ be a $C^4$ bounded uniformly convex domain in $\rz^n$. Where $f\in C^{2}(\overline{\Omega})$ is positive function and $\varphi \in C^{3}(\overline{\Omega})$. Then there exists a unique $k$ admissible solution $u \in C^{3,\alpha}(\overline{\Omega})$ of the boundary value problem,
\begin{equation}\label{MQ}
\left\{
\begin{array}{rccl}
\s_k (D^2 u) &=& f(x) & \hbox{ in } \Omega,\\
u_\nu&=&-u + \varphi(x)  &\hbox{ on } \partial \Omega.\\
\end{array}
\right.
\end{equation}
\end{theorem}

{\bf{Remark 1:}} For simplicity we only states this particular form of existence theorem,  due to the $C^0$ estimate is easy to handle in this case while we do not want to emphasize $C^0$ estimate in this paper (see \cite{LTU} for more general cases).

 Hessian equation is an important nonlinear elliptic partial differential equation.
It  appears naturally in classical geometry, conformal geometry and
K\"ahler geometry. Now let us brief recall some history and
development for this equation, for more detail please see the
paper by Wang~\cite{W}.

First for the Hessian equation on $R^n$, its Dirichlet boundary
value problem
\begin{equation}\label{Dirichlet}
\left\{
\begin{array}{rccl}
\s_k (D^2 u) &=& f(x) & \hbox{ in } \Omega,\\
u&=& \varphi(x)  &\hbox{ on } \partial \Omega,
\end{array}
\right.
\end{equation}
was studied by
Caffrelli-Nirenberg-Spruck~\cite{CNS2},
Ivochkina~\cite{Ivo85} and Trudinger~\cite{T1}. Chou-Wang~\cite{CW} got the Pogorelov
type interior estimates and the existence of variational solution.
Trudinger-Wang~\cite{TW} developed a Hessian measure theory for
Hessian operator.

  For the  curvature equations in classical geometry,  the existence of hypersurfaces with
prescribed Weingarten curvature was studied by  Pogorelov~\cite{Po75}, Caffarelli-Nirenberg-Spruck~\cite{CNS4, CNS5}, Guan-Guan~\cite{GG}, Guan-Ma~\cite{GM} and the later work by Sheng-Trudinger-Wang~\cite{STW1}. The Hessian equation on Riemannian manifolds  was also studied by
Y.Y.Li~\cite{LiY90}, Urbas~\cite{U02} and Guan~\cite{GB}. In recent years the Hessian type equation also appears in conformal geometry, which started from
Chang-Gursky-Yang~\cite{CGY} and the related development by (\cite{GW}, \cite{LL}, \cite{GV}, \cite{STW},\cite{GeW}). In K\"ahler geometry, the Hessian equation was studied by Hou-Ma-Wu~\cite{HMW} and Dinew-Kolodziej ~\cite{DK}.

  The  Yamabe problem on manifolds with boundary was first studied by Escobar~\cite{E}, he shows that (almost) every compact Riemannian manifold $(M,g)$ is conformally equivalent to one of constant scalar curvature, whose boundary is minimal.  The problem reduces to solving the semilinear elliptic critical Sobolev exponent equation  with the Neumann boundary condition. It is naturally, the Neumann boundary value problem for Hessian type equations also appears in the fully nonlinear Yamabe problem for manifolds with boundary, which is to find a conformal metric $\hat{g} = \exp(-2u) g$ such that the $k$-th elementary symmetric function of eigenvalues of Schouten tensor is constant and with the constant mean curvature on the boundary of manifold. See for  Jin-Li-Li~\cite{JLL}, Chen~\cite{SC} and  Li-Luc~\cite{LL1,LL2}, but in all these papers they need to impose the manifold  are  umbilic or total geodesic boundary for $k \geq 2$, which are more like the  condition  in Trudinger~\cite{Tru} that the domain is ball.

  The Neumann or oblique derivative problem on linear and quasilinear elliptic equations was widely studied for a long time, one can see the recent book written by Lieberman~\cite{Lieb13}. Especially for the mean curvature equation with prescribed contact angle boundary value problem,  Ural'tseva \cite{Ur73}, Simon-Spruck \cite{SS76} and Gerhardt \cite{Ger76} got the boundary gradient estimates and the corresponding existence theorem. Recently in \cite{MX}, Ma-Xu got the boundary gradient estimates and the corresponding existence theorem for the Neumann boundary value problem on mean curvature equation. For related results on the Neumann or oblique derivative problem for some class fully nonlinear elliptic equations can be found in Urbas~\cite{Urbas, Urbas2}.

     We give a brief description of our procedures and ideas to this problem. By the standard theory of  Lieberman-Trudinger~\cite{LT86} (see also \cite{LionT}, \cite{Lieb13}), it is well known that the solvability of the Hessian equations with Neumann boundary value can be reduced to the priori global second order derivative estimates. We have done $C^1$ estimate (jointed with J.J. Xu) in \cite{MQX} a year ago, there we constructed a suitable auxiliary function and use particular coordinate to let the estimate computable. For $C^2$ estimate, we first reduce the global estimate to the boundary double normal derivative, this estimate also plays an important role in our boundary double normal estimate. The main difficulty lies to construct the barrier functions of $u_{\nu}$. The Neumann boundary condition will bring us a trouble term as "$\sum\limits_{ijk} F^{ij}u_{ik}D_j \nu^{k}$". Motivated by Lions-Trudinger-Urbas \cite{LTU},
      Trudinger \cite{Tru},  Ivochkina- Trudinger-Wang~\cite{ITW} and Urbas \cite{Urbas}, we introduce a new barrier function,
  then we can extract a good term and  control this trouble term. For $C^0$ estimate, we deal with a particular form of $f$ and $\varphi$ as in \cite{Tru} for simplicity.

The rest of the paper is organized as follows. In section 2, we first give the definitions and some notations. We get the $C^0$ and $C^1$, which was obtained by Trudinger~\cite{Tru} and Ma-Qiu-Xu~\cite{MQX}. In section 4, we obtain the $C^2$ estimates, which is the main estimates in this paper. In last section 5, we prove the main Theorem~\ref{main}.

{\bf Acknowledgment:}  The both authors would
like to thank the helpful discussion and encouragement from Prof. X.-J. Wang. The first author would also thank Prof. P. Guan, Prof. N. Trudinger and Prof. J. Urbas for their interesting and encouragement. Research of the first author was supported by  NSFC. The second author was supported by the grant from USTC. Both authors was supported by Wu Wen-Tsun Key Lab.

\section{Preliminaries}

  In this section, we introduce the admissible solution and some element properties for $k$-th elementary symmetric function.
\begin{definition}\label{DF2.1}
For any $k= 1,2,\cdots,n$, and $\lambda = (\lambda_1,\cdots,\lambda_n) \in \rz^n$ we set
\begin{equation}\label{}
  \sigma_k(\lambda) = \sum_{1\leq i_1 < \cdots < i_k \leq n} \lambda_{i_1} \lambda_{i_2} \cdots \lambda_{i_k}.
\end{equation}
We denote by $\sigma_k(\lambda | i)$ the symmetric function with $\lambda_i = 0$ and $\sigma_k(\lambda|ij)$ the symmetric function with $\lambda_i = \lambda_j =0$.
Let $\lambda (D^2 u)$ be the eigenvalue of  $D^2 u$ and $\sigma_k(D^2 u) = \sigma_k(\lambda(D^2 u))$.
And we let
\begin{equation}\label{}
\Gamma_k = \{(\lambda_1, \cdots, \lambda_n)\in \rz^n | \indent \sigma_j (\lambda) >0 \indent \forall \indent j=1,\cdots,k\}.
\end{equation}

We say a function $u$ is $k$ admissible if $\lambda(D^2 u) \in \Gamma_k$.
\end{definition}

Denote
 $F^{ij}:=\frac{\partial\sigma_k(D^2u)}{\partial u_{ij}}$, $\mathcal{F}:=\sum\limits_{1\leq i\leq n}F^{ii}$.
 Sometimes we write the equation \eqref{EMQX} in the form
\begin{equation}\label{EV}
\widetilde{F}(D^2 u) := \sigma_k^{\frac{1}{k}}(D^2 u)=f^{\frac{1}{k}}=:\widetilde{f},
\end{equation}
and use the notation
\begin{equation}\label{}
\widetilde{F}^{ij} := \frac{\partial \widetilde{F}}{\partial u_{ij}}, \hbox{  }
\widetilde{F}^{ij,pq}:= \frac{\partial^2 \widetilde{F}}{\partial u_{ij} \partial u_{pq}}.
\end{equation}
$\sigma_k$ operator has following simple properties.
\begin{proposition}
\begin{equation}\label{8}
\sigma_k (\lambda) = \sigma_k(\lambda|i) + \lambda_i \sigma_{k-1}(\lambda|i), \forall 1\leq i \leq n,
\end{equation}
\begin{equation}\label{in4}
F^{ij} u_{ij}= k \sigma_k,
\end{equation}

and
\begin{equation}\label{in5}
 \mathcal{F}=(n-k+1)\sigma_{k-1}.
\end{equation}
\end{proposition}
\begin{proof}
See \cite{Lie}.
\end{proof}
\begin{proposition}
If $\lambda \in \Gamma_k$, then we have
\begin{equation}\label{inq11}
  \sigma_h(\lambda|i) >0, \indent \forall h<k \indent and \indent 1 \leq i \leq n,
\end{equation}
and $\sigma_k^{\frac{1}{k}}$  is  a  concave  function in  $\Gamma_k$.
\end{proposition}
\begin{proof}
See \cite{Lie}.
\end{proof}
The following proposition is so called MacLaurin inequality.
\begin{proposition}
For $\lambda \in \Gamma_k$ and $k\geq l \geq 1$, we have
\begin{equation}\label{Mac}
  [\frac{\sigma_k (\lambda)}{C^k_n}]^{\frac{1}{k}} \leq [\frac{\sigma_l(\lambda)}{C^l_n}]^{\frac{1}{l}}.
\end{equation}
Moreover,
\begin{equation}\label{inqu10}
  \sum_{1}^n \frac{\partial \sigma_k^{\frac{1}{k}}(\lambda)}{\partial \lambda_i} \geq [C^k_n]^{\frac{1}{k}}.
\end{equation}
\end{proposition}
\begin{proof}
See \cite{Lie}.
\end{proof}

\begin{proposition}
Let $\lambda \in \Gamma_k$. Suppose that
\begin{equation*}\label{}
  \lambda_1 \geq \cdots \geq \lambda_k \geq \cdots \geq \lambda_n,
\end{equation*}
then we have
\begin{equation}\label{inqu12}
  \lambda_1 \sigma_{k-1}(\lambda|1) \geq  \frac{k}{n} \sigma_k(\lambda),
\end{equation}
for $\forall i <k$
\begin{equation}\label{inqu13}
  \sigma_{k-1}(\lambda|i) \geq \sigma_{k-1}(\lambda|k) \geq c(n,k) \sigma_{k-1}(\lambda)>0,
\end{equation}
and
\begin{equation}\label{inqu14}
  \lambda_1 \geq \lambda_2 \geq \cdots \geq \lambda_k > 0.
\end{equation}
\end{proposition}
\begin{proof}
See  \cite{LT} for these inequalities, one can also see \cite{HMW} for the first inequality. The third one can be induced by the first inequality and \eqref{inq11}.
\end{proof}

\section{$C^0$ and $C^1$ estimates}
In this section we get the a priori bounded estimates and gradient estimates for the $k$- admissible solution of the equation \eqref{MQ}. For the $C^0$ estimates, which was gotten by Trudinger~\cite{Tru}.
\begin{theorem}\label{C^0}\cite{Tru}
Let $\O \subset \rz^n$ be a bounded $C^1$ domain, and $\nu$ is the outer unit normal vector of $\p \O$. Suppose $u \in C^2(\bar{\O})\bigcap C^3(\O)$ is an $k$ -admissible solution of the following Neumann boundary problems of Hessian equation
\begin{equation}
\left\{
\begin{array}{rccl}
\s_k (D^2 u) &=& f(x) & \hbox{ in } \Omega,\\
u_\nu&=&-u + \varphi(x)  &\hbox{ on } \partial \Omega.\\
\end{array}
\right.
\end{equation}
Then
\begin{equation}\label{}
  \sup\limits_{\overline{\Omega}}|u| \leq M_0,
\end{equation}
where $M_0$ depends on $k$, $n$, $diam \Omega$, $\varphi$, $\sup f$.
\end{theorem}
\begin{proof}
Taking $o\in \Omega$ and let us consider $u-A|x|^2$.  Fixing $A$ large depend on $k$, $n$ and $\sup f$ so that we have
\begin{eqnarray}
   F[D^2 u]=f \leq  F[D^2 (A|x|^2)].
\end{eqnarray}
Comparison principle tells us $u-A|x|^2$ attains its minimum point at $x_0$ on the boundary.
\begin{equation}\label{}
  0 \geq (u-A|x|^2)_{\nu}(x_0) = - u + \varphi-2A x \cdot \nu.
\end{equation}
Similarly we consider $u$ which attains its maximum on the boundary.
Then we get
\begin{equation}\label{}
  \inf\limits_{\partial \Omega} \varphi - 2 A \text{diam} \Omega \leq u \leq \sup\limits_{\partial \Omega} \varphi.
\end{equation}
\end{proof}

The gradient estimate was done in \cite{MQX}, since that paper was written in Chinese, for completeness we contain its proof in this section.
We set
\begin{align*}
 d(x)=\texttt{dist}(x,\partial \Omega),
 \end{align*}
 and
\begin{align*}
 \Omega_\mu:=&\{{x\in\Omega:d(x)<\mu}\}.
 \end{align*}
Then it is well known that there exists a positive constant $1 \geq \widetilde{\mu}>0$ such that $d(x) \in C^4(\overline \Omega_{\widetilde{\mu}})$. As in Simon-Spruck \cite{SS76} or Lieberman \cite{Lieb13} (in page 331),  we can extend $\nu$ by $\nu= -D d$ in $\Omega_{\widetilde{\mu}}$ and note that  $\nu$ is a $C^2(\overline \Omega_{\widetilde{\mu}})$ vector field. As mentioned in  the book \cite{Lieb13}, we also have the following formulas

\begin{align}\label{2.1}
\begin{split}
|D\nu|+|D^2\nu|\leq& C_0(n,\Omega) \quad\text{in}\quad \Omega_{\widetilde{\mu}},\\
 \sum_{1\leq i\leq n}\nu^iD_j\nu^i=0,  \sum_{1\leq i\leq n}\nu^iD_i\nu^j=&0, \,|\nu|=1 \quad\text{in} \quad\Omega_{\widetilde{\mu}}.
\end{split}
\end{align}
As in \cite{Lieb13}, we define
 \begin{align}\label{2.2}
\begin{split}
c^{ij}=&\delta_{ij}-\nu^i\nu^j  \quad \text{in} \quad \Omega_{\widetilde{\mu}},
\end{split}
\end{align}
 and for a vector $\zeta \in R^n$, we write $\zeta'$ for the vector with $i$-th component $ \sum_{1\leq j\leq n}c^{ij}\zeta^j$. Then we have
  \begin{align}\label{2.3}
\begin{split}
|D'u|^2=& \sum_{1\leq i,j\leq n}c^{ij}u_iu_j.
\end{split}
\end{align}
We first state an useful lemma from ~\cite{CW}.
\begin{lemma} (Chou-Wang)~\cite{CW}\label{CW1} If $u$ is $k$ -admissible and $u_{11} < -\frac{h^{\prime}|D u|^2}{128}$, here $h^{\prime}$ is any positive function. Then
\begin{equation}\label{inq5}
  \frac{1}{n-k+1} \mathcal{F} \leq F^{11},
\end{equation}
and
\begin{equation}\label{inq10}
  \mathcal{F} \geq C^{k-1}_{n-1} [\frac{ h^{\prime}}{128 C^k_{n-1}}]^{k-1}|D u|^{2k-2}.
\end{equation}
\end{lemma}
 To state the gradient estimate on Neumann problems, we need first recall an interior estimate in \cite{CW}.\\
\begin{lemma} (Chou-Wang)~\cite{CW}\label{CW2} Let $\O \subset \rz^n$ be a bounded domain. Suppose $u \in C^3(\O)$ is a $k$-admissible solution of Hessian equation
\begin{equation}\label{}
  \sigma_k(D^2 u) = f(x,u) \indent in  \indent \O
\end{equation}
satisfying $|u|\leq M_0$. If $f \in C^2(\bar{\O}\times [-M_0,M_0])$ satisfies the conditions that there exist positive constant $L_1$ such that
\begin{equation}\label{}
\begin{array}{rccl}
  f(x,z)\geq 0 & in &  \bar{\O}\times [-M_0,M_0],\\
  |f(x,z)| + |f_x(x,z)| + |f_z(x,z)| \leq L_1 & in & \bar{\O}\times [-M_0,M_0].
  \end{array}
\end{equation}
Then for $\forall$  $\O^{'} \subset\subset \O$, it has
\begin{equation}\label{}
  \sup_{\O^{'}}|D u| \leq \widetilde{M}_1,
\end{equation}
where $\widetilde{M}_1$ is a positive constant which depends on $n ,k, M_0, dist(\O^{'},\p \O), L_1$.
\end{lemma}

Now we get the global gradient estimate which was done in \cite{MQX}.
\begin{theorem}\label{MQX}
Let $\O \subset \rz^n$ be a bounded $C^3$ domain, and $\nu$ is the outer unit normal vector of $\p \O$. Suppose $u \in C^2(\bar{\O})\bigcap C^3(\O)$ is an $k$-admissible solution of the following Neumann boundary problems of Hessian equation
\begin{equation}
\left\{
\begin{array}{rccl}
\s_k (D^2 u) &=& f(x,u) & \hbox{ in } \Omega,\\
u_\nu&=&\varphi(x,u)  &\hbox{ on } \partial \Omega,\\
\end{array}
\right.
\end{equation}
satisfying $|u|\leq M_0$, where $f,\varphi$ are given functions defined on $\bar{\O}\times [-M_0,M_0]$. If $f, \varphi$ satisfy the conditions: $\exists$ positive constants $ L_1, L_2$ such that
\begin{equation}\label{}
\begin{array}{rccl}
  f(x,z)>0 & in &  \bar{\O}\times [-M_0,M_0],\\
  |f(x,z)| + |f_x(x,z)| + |f_z(x,z)| \leq L_1 & in & \bar{\O}\times [-M_0,M_0],\\
  |\varphi(x,z)|_{{C^3}{(\bar{\O}\times [-M_0,M_0])}} \leq L_2.
  \end{array}
\end{equation}
Then there exists a small positive constant $\mu_0$ which depends only on $n, k, \O, \\M_0, L_1, L_2$ such that
\begin{equation}\label{conclusion}
  \sup_{\bar{\O}_{\mu_0}}|D u| \leq \max\{\widetilde{M}_1, \widetilde{M}_2\},
\end{equation}
where $\widetilde{M}_1$ is a positive constant depending only on $n,k, \mu_0, M_0, L_1$, which is from the interior gradient estimates; $\widetilde{M}_2$ is a positive constant depending only on $n, k, \O, \mu_0, M_0, L_1, L_2$.
\end{theorem}

\subsection{Proof of Theorem \ref{MQX}}
\begin{proof}
We consider the auxiliary function
\begin{equation}\label{G}
  G(x) := \log |D w|^2 + h(u) + g(d),
\end{equation}
where
\begin{equation}\label{eq2}
w(x) := u(x) + \varphi(x,u) d(x);
\end{equation}

\begin{equation}\label{eq1}
 h(u) := -\log (1+4M_0- u);
\end{equation}
and
\begin{equation}\label{}
g(d):= \alpha_0 d,
\end{equation}
in which $\alpha_0$ large to be chosen later.\\
By \eqref{eq1} we have
\begin{align}\label{}
 - log (1+ 5M_0) &\leq h \leq -log(1+3M_0), \\
  \frac{1}{1+5M_0} &\leq h^{\prime}  \leq \frac{1}{1+3M_0}, \\
  \frac{1}{(1+5M_0)^2} & \leq h^{\prime\prime} \leq \frac{1}{(1+3M_0)^2}.
\end{align}
By \eqref{eq2} we have
\begin{align}\label{}
  w_i =& u_i + (\varphi_i + \varphi_z u_i)d + \varphi d_i. \label{eqn7}
\end{align}
If we assume that $|D u| > 8 n L_2 $ and $\mu_0 \leq \frac{1}{2 L_2}$, it follows from \eqref{eqn7} that
\begin{equation}\label{}
  \frac{1}{4}|D u| \leq |D w| \leq 2|D u|.
\end{equation}\\
These inequalities will be used below.\\
We assume that
$G(x)$ attains its maximum at $x_0 \in \overline \Omega_{\mu_{0}}$, where $0<\mu_0<\widetilde{\mu} \leq 1$ is a sufficiently small number which we shall decide it  later.

Now we divide three cases to complete the proof of  Theorem~\ref{MQX}.

Case I: If $G(x)$ attains its maximum at $x_0 \in \partial\Omega$, then we shall use the Hopf Lemma to get the bound of $G(x_0)$.

Case II:  If $G(x)$ attains its maximum at $x_0 \in \Omega_{\mu_0}$, in this case for the sufficiently small constant $\mu_0>0$,  then we can use the maximum principle to get the bound of $G(x_0)$.

Case III: If $G(x)$ attains its maximum at $x_0 \in\partial\Omega_{\mu_0}\bigcap\Omega$, then we shall get the estimates of $|D u|(x_0)$ via the standard interior gradient bound as in \cite{CW}. Which in turn give the bound for $G$ at point $x_0$.\\
Since $G(x)\leq G(x_0)$, we get the bound of $G$, which in turn give the bound of $|\nabla u|$ in ${\bar{\O}}_{\mu_0}$.\\
Now  all computations work at the point $x_0$. We use Einstein's summation convention. All repeated indices come from $1$ to $n$.
\subsection{Case I: boundary estimates}
If maximum of $G$ is attained on the boundary, at the maximum point we have
\begin{equation}\label{eqn1}
  0 \leq G_\nu = \frac{ |Dw|^2_p \nu^p}{|Dw|^2} - g^{\prime} + h^{\prime} u_{\nu}.
\end{equation}
We have decomposition $|Dw|^2 = |D^{\prime} w|^2 + w_{\nu}^2$.
Becuse $w_{\nu} = u_{\nu} + D_{\nu} \varphi d - \varphi = 0$ on the boundary,
so we have
\begin{align}
  |Dw|^2_p \nu^p =&  C^{ij}_p w_{i} w_{j}\nu^p+ 2 C^{ij} w_{ip} w_{j}\nu^p+ 2 w_{\nu} D_p w_{\nu} \nu^{p},\nonumber\\
 =&  C^{ij}_p w_{i} w_{j}\nu^p
 + 2 C^{ij} (u_{ip}+ D_{ip} \varphi d + D_i \varphi d_p + D_p \varphi d_i \nonumber\\& + \varphi d_{ip})w_j \nu^p,\nonumber\\
 =&   C^{ij}_p w_{i} w_{j}\nu^p + 2 C^{ij} u_{i\nu}w_j - 2C^{ij}D_i\varphi w_j + 2C^{ij}D_p \varphi \nu^p d_i w_j \nonumber\\&+ 2C^{ij} \varphi d_{ip} w_j \nu^p.\label{eqn2}
\end{align}
On the other hand, take tangential derivative to the Neumann boundary condition:
\begin{align*}\label{}
  C^{pq} D_q(u_i \nu^i) = C^{pq} D_q \varphi,
\end{align*}
then we have
\begin{align}\label{}
  C^{pq} u_{q \nu} + C^{pq} u_i D_q \nu^{i} = C^{pq} D_q \varphi.\label{eqn3}
\end{align}

Then contracting \eqref{eqn3} with $w_p$, and inserting it into \eqref{eqn2}, we can cancel the term with the second derivative of $u$,
\begin{align}\label{}
  |Dw|^2_p \nu^p \leq & C(n,\O, L_2)|D w|^2 + C(n,\O,L_2)|D w|.
\end{align}
So we choose $\alpha_0 = 2C+ \frac{L_2}{1+3M_0} +1 $, such that
\begin{align}\label{inq2}
  0 \leq G_{\nu} \leq& -\alpha_0 + C+ \frac{C}{|D w|}+ h^{\prime} |\varphi|_{C^0} \nonumber\\
  \leq & -C+ \frac{C}{|D w|}.
\end{align}
Thus we have estimate $|D w|(x_0) \leq 1$, and $G(x_0) \leq -log(1+3M_0)+2C+ \frac{L_2}{1+3M_0} +1$.

\subsection{Case II: Near boundary estimates}
If G attains its maximum in $\O_{\mu_0}$. We take the first derivatives and second derivatives to the auxiliary function:
\begin{align}
  0 = G_i & = \frac{2 \sum^{n}\limits_{p=1} w_p w_{pi }}{|D w|^2} + g^{\prime} D_i d + h^{\prime} u_i, \label{eqn14} \\
  G_{ij} =& \frac{ \sum^{n}\limits_{p=1} 2 w_{pj} w_{pi} + 2 w_{p} w_{pji}}{|D w|^2} -
  \frac{4 \sum^{n}\limits_{p,q=1} w_p w_{pi} w_q w_{qj}}{|D w|^4}\nonumber\\
  &+ g^{\prime \prime} D_{i} d D_{j} d + g^{\prime} D_{ij} d + h^{\prime \prime} u_i u_j+ h^{\prime} u_{ij}.
\end{align}
Because $F^{ij}(D^2 u) > 0$ if we assume u is $k$- admissible solution.
At maximum point of G, we get
\begin{align}\label{}
  0 \geq F^{ij} G_{ij} =& \frac{2 \sum^{n}\limits_{p=1}F^{ij}w_{pi}w_{pj}}{|D w|^2}+\frac{2\sum^{n}\limits_{p=1} F^{ij} w_p w_{pij}}{|D w|^2} - \frac{4 \sum^{n}\limits_{p,q=1}F^{ij}w_p w_{pi} w_q w_{qj}}{|D w|^4}\nonumber \\
  &+g^{\prime \prime} F^{ij} D_i d D_j d + g^{\prime} F^{ij} D_{ij} d \nonumber\\ &+h^{\prime \prime} F^{ij} u_i u_j+ h^{\prime} F^{ij} u_{ij}.
\end{align}
Recalling $w = u + \varphi d$, its second derivatives is
\begin{align}\label{}
  w_{ij}= & u_{ij} + (\varphi_{ij}+ \varphi_{i z}u_j+\varphi_{zj}u_i + \varphi_{zz}u_i u_j + \varphi_z u_{ij})d \label{eqn4}\nonumber\\
  &+(\varphi_i + \varphi_z u_i)d_j+\varphi_j d_i +\varphi_z u_j d_i+ \varphi d_{ij}.
\end{align}
$w_{ij}$ has relation with $u_{ij}$ that
\begin{equation}\label{}
  w_{ij} \leq (1+\varphi_z d)u_{ij} + C(L_2, n)\mu_0 |D u|^2 + C(L_2, n) |D u| + C(L_2,n),
\end{equation}
and
\begin{equation}\label{}
  w_{ij} \geq (1+\varphi_z d)u_{ij} -C(L_2, n)\mu_0 |D u|^2 - C(L_2, n) |D u| - C(L_2,n).
\end{equation}
Differential $ w_{ij}$ again,
\begin{equation}\label{eqn6}
\begin{split}
  w_{ijp} =& u_{ijp} + (\varphi_{ijp}+ \varphi_{ij z} u_p +\varphi_{i z p}u_j + \varphi_{i zz}u_p u_j + \varphi_{i z} u_{jp}+\varphi_{z j p} u_i\\
  & +\varphi_{zz j} u_p u_i +\varphi_{z j} u_{ip}+ \varphi_{zzp}u_i u_j +\varphi_{zzz}u_p u_i u_j + \varphi_{zz}u_{ip}u_j \\
  &+\varphi_{zz} u_i u_{jp} +\varphi_{z p} u_{ij} + \varphi_{zz} u_p u_{ij} + \varphi_z u_{ijp})d \\
  &+ (\varphi_{ij}+ \varphi_{i z}u_j+\varphi_{zj}u_i + \varphi_{zz}u_i u_j + \varphi_z u_{ij})d_p \\
  &+(\varphi_{ip}+\varphi_{iz} u_p+\varphi_{zp}u_i + \varphi_{zz}u_p u_i + \varphi_z u_{ip})d_j\\
  &+(\varphi_i + \varphi_z u_i)d_{jp}+ \varphi_{jp} d_i + \varphi_{j z} u_p d_i + \varphi_j d_{ip}\\
  & +\varphi_{zp}u_j d_i + \varphi_{zz} u_p u_j d_i+\varphi_z u_{jp} d_i\\
  &+ \varphi_{z}u_j d_{ip}+ \varphi_p d_{ij} + \varphi_z u_p d_{ij} + \varphi d_{ijp}.
\end{split}
\end{equation}
Now we choose coordinate at $x_0$ such that $|\nabla w| = w_1$ and $(u_{ij})_{2\leq i,j \leq n}$ is diagonal.\\
So from \eqref{eqn7} and \eqref{eqn14}, we have for $i=1$,
\begin{align}\label{}
  u_1 =& \frac{w_1 - \varphi_1 d - \varphi d_1}{1+\varphi_z d},\\
  w_{11} = &-\frac{1}{2} (g^{\prime} d_1 + h^{\prime}u_1) w_1, \label{eqn5}
\end{align}
and for $2 \leq i \leq n$,
\begin{align}\label{}
  u_i =& \frac{-\varphi_i d - \varphi d_i}{1+\varphi_z d}, \\
  w_{1i} =& -\frac{1}{2}(g^{\prime}d_i + h^{\prime} u_i) w_1,\label{in1}
\end{align}
here we assume $\mu_0 \leq \mu_1 := \frac{1}{2 L_2}$, such that $\frac{3}{2} \geq 1+\varphi_z d \geq \frac{1}{2}$.\\
Suppose that $|D u|(x_0)> M_1 := 64 n L_2 $, we have for $i\geq 2$,
\begin{equation}\label{}
|u_i| \leq \frac{1}{16n} |D u|,
\end{equation}
and
\begin{equation}\label{}
 u_1 \geq \frac{1}{2}|D u|.
 \end{equation}

 Moreover,
 \begin{equation}\label{}
 |D u|(x_0) \geq M_2: =  32n(1+5M_0)\alpha_0+128C +(1+5M_0)+1
\end{equation}
 implies
\begin{equation}\label{}
 |g^{\prime}d_i| \leq \frac{h^{\prime} u_1}{16n}.
\end{equation}
  So from \eqref{eqn4} and \eqref{eqn5} we get the key fact that
\begin{equation}\label{inq3}
  u_{11} \leq - \frac{1}{128}  h^{\prime} |D u|^2 < 0,
\end{equation}
here we assume that $\mu_0 \leq \mu_2 : =\frac{1}{64 C (1+5M_0)}$.\\
For $i\geq 2$, we have
\begin{equation}\label{in2}
  |w_{1i}| \leq \frac{ h^{\prime}|D w|^2}{32n},
\end{equation}
and
\begin{equation}\label{inq4}
  |u_{1i}| \leq (C \mu_0+\frac{1}{1+3M_0}) |D u|^2+2C |D u|.
\end{equation}
Then we continue to compute $F^{ij}G_{ij}$. By using \eqref{eqn7}, \eqref{eqn4} and \eqref{eqn6} it  follows that
\begin{align}\label{}
  F^{ij}G_{ij} \geq &  - C(n, k,L_2, \O)\mu_0 \mathcal{F} |D u|^2 + \frac{2F^{ij}u_{ij1}(1-\varphi_z d)}{w_1}\nonumber \\
  &- \frac{4F^{ij}(\varphi_{i z}u_{j1}d + \varphi_{zz}u_{i1}u_j d+ \varphi_z u_{i1}d_j)}{w_1}\nonumber\\
   &- \frac{2 F^{ij} u_{ij}[(\varphi_{zp}+\varphi_{zz} u_p )d+\varphi_z d_p]}{w_1}-2\frac{F^{ij}w_{1i}w_{1j}}{w^2_1}\nonumber\\
  & -C(n, k, L_2, \alpha_0, \O) \mathcal{F}|D u| + h^{\prime\prime} F^{11}u^2_1 + h^{\prime} F^{ij}u_{ij}.
\end{align}
The equation \eqref{EMQX} is $k$- homogenous, and differentiating it gives
\begin{align}
  F^{ij}u_{ij} =& k f, \label{eqn8}\\
  F^{ij}u_{ij1} = & f_1 + f_u u_1.\label{eqn9}
\end{align}
We obtain from \eqref{in2}, \eqref{inq4}, \eqref{eqn8}, and \eqref{eqn9} that
\begin{align}\label{}
  F^{ij}G_{ij} \geq & - C(n, k,L_2,M_0, \O)\mu_0 \mathcal{F} |D u|^2 +h^{\prime\prime} F^{11}u^2_1-\frac{( h^{\prime})^2\mathcal{F} |Dw|^2}{32}\nonumber\\
  &- C(n, k,L_2,L_1,M_0,\alpha_0, \O) \mathcal{F} |D u|-C(L_1,n,L_2) .
\end{align}

\eqref{inq5} tells us if
$\mu_0 \leq \mu_3 : = \frac{1}{32 C(1+5M_0)^2(n-k+1)}$ small, we get
\begin{equation}\label{inq16}
  \frac{h^{\prime\prime}F^{11}u^2_1}{8} \geq C \mu_0\mathcal{F} |D u|^2.
\end{equation}
By definition of $h$, we have $h^{\prime\prime}=(h^{\prime})^2$.
Thus from \eqref{inq5}
\begin{equation}\label{inq17}
  \frac{h^{\prime\prime}F^{11}u^2_1}{8} \geq \frac{( h^{\prime})^2\mathcal{F} |Du|^2}{32}.
\end{equation}
If we assume further $|Du|^2(x_0) \geq M_3: = 32(n-k+1)(1+5M_0)^2C$, we get
\begin{equation}\label{inq18}
  \frac{h^{\prime\prime}F^{11}u^2_1}{8} \geq  C \mathcal{F} |D u|.
\end{equation}
From above estimates \eqref{inq5}, \eqref{inq16}, \eqref{inq17}, and \eqref{inq18}, we obtain
\begin{equation}\label{inq6}
  0 \geq F^{ij}G_{ij} \geq \frac{ h^{\prime \prime} \mathcal{F} |Du|^2}{32(n-k+1)} - C.
\end{equation}
Finally, inequality \eqref{inq10} in the Lemma~\ref{CW1} implies that
\begin{equation}\label{inq21}
  0 \geq \frac{ h^{\prime \prime} \mathcal{F} |Du|^2}{32(n-k+1)} - C > 0,
\end{equation}
provided that $|Du|(x_0) \geq M_4 := \frac{32(n-k+1)(1+5M_0)^2 C[(1+5M_0)128C^k_{n-1}]^{k-1}}{C^{k-1}_ {n-1}}+1$.\\
Inequality \eqref{inq21} is a contradiction.\\
We conclude that if $\mu_0 = \min\{\widetilde{\mu},\mu_1,\mu_2,\mu_3\}$, we have the estimate
\begin{equation}\label{inq20}
|Du|(x_0) \leq \max\{M_1,M_2,M_3,M_4\}.
\end{equation}
Thus we get the estimate of $G(x_0)$. \\
Because $G$ attains its maximum at $x_0$ and $h,g$ is bounded from below, the gradient estimate of $u$ follows the above three cases.

\end{proof}

\section{$C^2$ priori estimates}
We come now to the a priori estimates of second derivative necessary for our existence theorem. For these bounds we restrict attention to the following problem
\begin{equation}\label{NP}
\left\{
\begin{array}{rccl}
 \sigma_k(D^2 u)&=& f(x,u)&
\hbox{ in } \Omega \subset\mathbb{R}^n,\\
u_{\nu}&=&\varphi(x,u)  &\hbox{ on } \partial \Omega.\\
\end{array}
\right.
\end{equation}

\begin{theorem}\label{C^2}
Let $\Omega$ be a bounded $C^{4}$ uniformly convex domain in $\mathbb{R}^n$, $\nu$ is the outer unit normal vector of $\p \O$. If $u\in C^4(\Omega)\cap C^3(\overline{\Omega})$ a $k$- admissible solution of Neumann problem \eqref{NP}. Where $f\in C^{2}(\overline{\Omega}\times\mathbb{R})$ is positive and $\varphi \in C^{3}(\overline{\Omega}\times\mathbb{R})$ is non-increasing in $z$. Then we have
\begin{equation}
\sup\limits_{\overline{\Omega}} |D^2 u| \leq C,
\end{equation}
where $C$ depends only on $n$, $k$, $||u||_{C^1(\overline{\Omega})}$, $|| f||_{C^2(\overline{\Omega}\times[-M_0,M_0])}$, $\min f$,\\ $||\varphi||_{C^{3}(\overline{\Omega}\times[-M_0,M_0])}$ and convexity of $\Omega$, where $M_0=\sup\limits_{\Omega}|u|$.
\end{theorem}

It is well known that it is easy to get the estimates for second tangential-normal derivative of the solution on the boundary.
 We here follow the same line as in Lions-Trudinger-Urbas \cite{LTU} with minor changes.
\begin{lemma}\label{tn}
Denoting the tangential direction $\tau$ at any point $y\in \partial \Omega$, we have
\begin{equation}
|D_{\tau \nu} u (y)| \leq C,\label{lemma4.1}
\end{equation}
where the constant $C$ only depends on $||u||_{C^1}$, $||\varphi||_{C^1}$ and $||\partial \Omega||_{C^2}$.
\end{lemma}
\begin{proof}
Taking tangential derivative to the boundary condition
\begin{equation}
u_{\nu} = \varphi,
\end{equation}
as in (\ref{eqn3}) we have

\begin{align}\label{}
  C^{ij} u_{j \nu} + C^{ij} u_l D_j\nu^{l} =  C^{ij} D_j \varphi.\label{eqn3a}
\end{align}
Take inner pruduct with $\tau^i$, it follows that
\begin{equation}\label{}
   \tau^i u_{li} \nu^l + u_l D_i \nu^{l}\tau^{i} = D_{i} \varphi \tau^{i}.
\end{equation}
So
\begin{equation}\label{}
  |u_{\tau \nu}| \leq |D_{i} \varphi \tau^{i} - u_l D_i \nu^{l}\tau^{i}| \leq C.
\end{equation}
\end{proof}

Now we again use the technique of Lions-Trudinger-Urbas~\cite{LTU},  we can reduce the second derivative estimates of the solution  to the boundary double normal derivative bounds.
\begin{lemma}\label{gl}
Let $ M= \sup\limits_{\partial\Omega} |u_{\nu\nu}|$.
Then
\begin{equation} \label{F1}
\sup\limits_{\overline{\Omega},\xi\in \mathbb{S}^{n-1}} u_{\xi\xi} \leq C_0 (1+M),
\end{equation}
where $C_0$ depends only on $||u||_{C^1}$,$||\varphi||_{C^3}$,
  $||\partial\Omega||_{C^4}$,$||f||_{C^2}$, $\min f$, and convexity of $\partial \Omega$.
\end{lemma}

\begin{proof}
We consider the function
\begin{equation}
v(x,\xi) := u_{\xi\xi} - v^{\prime}(x,\xi) + K_1 |x|^2 + K_2 |D u|^2,
\end{equation}
where $v^{\prime}(x,\xi):= 2 (\xi \cdot \nu)\xi^{\prime}\cdot (D \varphi -   u_l D \nu^{l})= a^l u_l  + b $,  $\xi^{\prime} = \xi - (\xi \cdot \nu)\nu$, $a^l = 2(\xi \cdot \nu)(\xi^{\prime l} \varphi_z  -   \xi^{\prime i} D_i \nu_l)$, and $b = 2 (\xi \cdot \nu)  \xi^{\prime l} \varphi_{x_l}$.
We compute
\begin{equation}
v_i = u_{\xi\xi i} -  D_i a^l u_l-  a^l u_{li} - D_i b + 2K_1 x_i + 2  \sum\limits_{l} K_2 u_l u_{li},
\end{equation}
and
\begin{eqnarray}\label{vf2}
v_{ij} =& u_{\xi\xi ij} -  D_{ij}a^l u_l - D_i a^l u_{lj}- D_j a^l u_{li} -  a^l u_{lij} - D_{ij}b \nonumber \\
 &+ 2K_1 \delta_{ij} + 2  K_2 \sum\limits_{l}  u_{li} u_{lj}+ 2 K_2 \sum\limits_{l}  u_l u_{lij}.
\end{eqnarray}

Taking first derivative of equation \eqref{EV}, we have
\begin{equation}\label{EVf1}
\widetilde{F}^{ij} u_{ijl} = \widetilde{f}_{x_l}+\widetilde{f}_{z}u_l.
\end{equation}

And we have from the concavity of $\sigma_k^{\frac{1}{k}} $
\begin{equation}\label{con}
\widetilde{F}^{ij} u_{ij \xi \xi} \geq \widetilde{F}^{ij} u_{ij \xi \xi} + \widetilde{F}^{ij,pq}u_{ij\xi}u_{pq\xi} = \widetilde{f}_{x_\xi x_\xi}+2\widetilde{f}_{x_\xi z} u_\xi + \widetilde{f}_{z} u_{\xi \xi}.
\end{equation}
Then we contract \eqref{vf2} with the $\widetilde{F}^{ij}$, using \eqref{con} and \eqref{EVf1},
\begin{eqnarray}
\widetilde{F}^{ij} v_{ij} =& \widetilde{F}^{ij}u_{\xi\xi ij} - \widetilde{F}^{ij}D_{ij}a^l u_l - 2  \widetilde{F}^{ij} u_{lj} D_i a^l-  \widetilde{F}^{ij} u_{lij} a^l  \nonumber  \\
&  - \widetilde{F}^{ij} D_{ij} b + 2 K_1 \sum\limits_{i}\widetilde{F}^{ii}  + 2 K_2 \sum\limits_{l} \widetilde{F}^{ij} u_{lj} u_{li} + 2 K_2 \sum\limits_{ijl} \widetilde{F}^{ij} u_{lij} u_l  \nonumber  \\
\geq &  - C_1 (||u||_{C^1},||\varphi||_{C^3},||\partial\Omega||_{C^4},
||f||_{C^2},\min f,K_2)
(\sum\limits_{i}\widetilde{F}^{ii}+1) \nonumber  \\
& +\widetilde{f}_z u_{\xi\xi}+ 2K_1 \sum\limits_{i}\widetilde{F}^{ii} + 2  K_2 \sum\limits_{l} \widetilde{F}^{ij} u_{li} u_{lj} - 2 \widetilde{F}^{ij} u_{lj} D_i a^l.
\end{eqnarray}
At interior maximum point, we assume $(u_{ij})$ is diagonal and $u_{11} \geq u_{22} \geq \cdots \geq u_{nn}$. So we have by \eqref{inqu12}
\begin{eqnarray}
  2 K_2\sum\limits_{i} \widetilde{F}^{ii}u^2_{ii} &\geq& 2 K_2 \sigma_{k}^{\frac{1}{k}-1}F^{11}u^2_{11} \nonumber\\
   &\geq& 2 K_2  \frac{\sigma_{k}^{\frac{1}{k}}}{n} u_{11}\nonumber\\
   & \geq & 2 K_2 \frac{\sigma_{k}^{\frac{1}{k}}}{n} u_{\xi\xi}.
\end{eqnarray}
We can assume $u_{\xi\xi} \geq 0$, otherwise we have the estimate \eqref{F1}. If we choose $K_2 \geq \frac{n|\widetilde{f}_z|}{2 \min \widetilde{f}} + 2$, we continue

\begin{eqnarray}
\sum\limits_{ij}\widetilde{F}^{ij}v_{ij} \geq & 2 \sum\limits_{i} \widetilde{F}^{ii} u_{ii}^2 - 2 C_2(||u||_{C^1},||\varphi||_{C^3},||\partial\Omega||_{C^3}) \sum\limits_{i}\widetilde{F}^{ii}|u_{ii}|  \nonumber\\
&+ 2K_1 \sum\limits_{i} \widetilde{F}^{ii} - C_1(\sum\limits_{i} \widetilde{F}^{ii} + 1) \nonumber\\
\geq & 2 \sum\limits_{i} \widetilde{F}^{ii} (|u_{ii}| - \frac{C_2}{2})^2 +(2 K_1-\frac{C_2^2}{2} - C_1)  \sum\limits_{i} \widetilde{F}^{ii} -C_1.
\end{eqnarray}
Now if we choose $K_1$ large, such that $K_1 \geq \frac{C_2^2}{2} + C_1$ and $K_1  (C_n^k)^{\frac{1}{k}}  > C_1$, by \eqref{inqu10} we have
\begin{equation}
\sum\limits_{ij}\widetilde{F}^{ij}v_{ij} > 0 .
\end{equation}
So $v(x,\xi)$ attains its maximum on $\partial \Omega$.\\
\noindent{\bf{Case a: $\xi$ is tangential.}}\\
We shall take tangential derivative twice to the boundary condition, first we rewrite (\ref{eqn3a}) as following
\begin{align}\label{}
u_{li}{\nu}^l = C^{ij}D_j \varphi - C^{ij}u_lD_{j}{\nu}^l + {\nu}^i{\nu}^j{\nu}^l u_{lj}.\label{eqn3aa}
\end{align}
So let's take tangential derivative (\ref{eqn3aa}) and we get
\begin{eqnarray}\label{}
C^{pq}D_{q}(u_{li}{\nu}^l) = C^{pq}D_{q}(C^{ij}D_j \varphi - C^{ij}u_lD_{j}{\nu}^l + {\nu}^i{\nu}^j{\nu}^l u_{lj}),\label{eqn3aaa}
\end{eqnarray}
it follows that
\begin{eqnarray*}\label{}
u_{lip}{\nu}^l = C^{pq}D_{q}(C^{ij}D_j \varphi - C^{ij}u_lD_{j}{\nu}^l + {\nu}^i{\nu}^j{\nu}^lu_{lj})+ {\nu}^p{\nu}^q{\nu}^lu_{liq} -C^{pq}u_{li}D_q{\nu}^l,
\end{eqnarray*}
and in above formula we take sum with $\xi^{i}\xi^{p}$, then we obtain
\begin{eqnarray}\label{}
u_{\xi\xi\nu} &=& -2  \xi^p \xi^i u_{li} D_p \nu^l- u_l \xi^p D_{ip} \nu^{l}\xi^{i} + u_{\nu\nu} \sum_{i} \xi^p D_{p} \nu^{i}\xi^{i}\nonumber\\
& &-\sum_{i} \xi^p\xi^i\nu^jD_p \nu^iD_{j}\varphi + {\varphi}_{z}u_{\xi\xi} + \xi^p\xi^i{\varphi}_{ip} \nonumber\\&&+ {\varphi}_{zz}u_{\xi}^2 + 2u_{\xi}{\xi}^i{\varphi}_{zi}.
\end{eqnarray}

So we have
\begin{eqnarray}\label{}
 u_{\xi\xi\nu} &\le & -2   \xi^p \xi^i u_{li} D_p \nu^l+ {\varphi}_{z}u_{\xi\xi} \nonumber\\&&+ C(||u||_{C^1},||\partial \Omega||_{C^3},||\varphi||_{C^2})+ C(||\partial  \Omega||_{C^2}) |u_{\nu \nu}|\nonumber\\
 &\leq & -2   \xi^p \xi^i u_{li} D_p \nu^l + C+ C |u_{\nu \nu}|.
\end{eqnarray}
Here in the second inequality we assume that $\varphi$ is non-increasing in $z$.\\
If we assume $\xi= e_1$, it is easy to get the bound for $u_{1i}(x_0)$ for $i \neq 1$ from the maximum of $v(x,\xi)$ in the $\xi$ direction. In fact, we can assume $\xi(t) = \frac{(1,t,0,\cdots,0)}{\sqrt{1+t^2}}$. Because $v(x,\xi)$ attains its maximum at $\xi(0)$. Then we have
\begin{eqnarray}
  0 &=& \frac{\partial v(x_0,\xi(t))}{\partial t}|_{t=0} \nonumber\\
    &=& 2 u_{ij}(x_0) \frac{d \xi^i(t)}{d t}|_{t=0} \xi^j(0) - \frac{\partial v^{\prime}(x_0,\xi(t))}{\partial t}|_{t=0}\nonumber\\
    &=& 2 u_{11}\frac{-t}{(1+t^2)^\frac{3}{2}}|_{t=0} + 2 u_{12}(\frac{1}{\sqrt{1+t^2}}+
    \frac{-t^2}{(1+t^2)^\frac{3}{2}})|_{t=0}-
    \frac{\partial v^{\prime}}{\partial t}|_{t=0}.
\end{eqnarray}
So we have
\begin{equation}\label{}
  |u_{12}| \leq C(||\varphi||_{C^1},||u||_{C^1},||\partial \Omega||_{C^2}).
\end{equation}
Similarly, we have for all $i \neq 1$,
\begin{equation}\label{}
  |u_{1i}| \leq C(||\varphi||_{C^1},||u||_{C^1},||\partial \Omega||_{C^2}).
\end{equation}
Due to $D_{1} \nu_{1} \geq \kappa >0$, we have
\begin{equation}\label{100}
  u_{\xi\xi\nu} \leq - 2 \kappa u_{\xi\xi} + C(1+|u_{\nu\nu}|).
\end{equation}
On the other hand, we have from the Hopf lemma, (\ref{lemma4.1})  and  $\sum\limits_i a^{i}\nu^i=0$,
\begin{eqnarray}
  0 & \leq & v_{\nu}\nonumber \\
   &=& u_{\xi\xi\nu} - D_{\nu}a^l u_l - a^l u_{l\nu} -b_{\nu}+ 2K_1(x\cdot \nu) + 2 K_2 \sum_l u_{l}u_{l\nu}\nonumber\\
   & \leq & u_{\xi \xi \nu} + C(||u||_{C^1}, ||\partial \Omega||_{C^2}, ||\varphi||_{C^2},K_1,K_2) + 2 K_2 \varphi u_{\nu\nu}.\label{101}
\end{eqnarray}
Combining (\ref{100}) and (\ref{101}), we therefore deduce
\begin{equation}\label{}
  u_{\xi \xi}(x_0) \leq C(1+| u_{\nu \nu}|(x_0)).
\end{equation}

{\bf{Case b: $\xi$ is non-tangential.}}\\
We write $\xi = \alpha \tau + \beta \nu$, where $\alpha = \xi \cdot \tau$, $|\tau|=1$, $\tau\cdot \nu =0$, $\beta=\xi\cdot\nu\neq0$ and $\alpha^2 +\beta^2 =1$.
\begin{eqnarray}
  u_{\xi\xi} &=& \alpha^2 u_{\tau\tau}+\beta^2 u_{\nu\nu}+2\alpha \beta u_{\tau \nu}  \nonumber\\
   &=& \alpha^2 u_{\tau\tau}+ \beta^2 u_{\nu\nu} + 2\alpha \beta(D_i \varphi \tau^i - u_{l} D_i \nu^l \tau^i).
\end{eqnarray}
By definition of $v(x,\xi)$, we have
\begin{eqnarray}
  v(x_0,\xi) &=& \alpha^2 v(x_0,\tau) + \beta^2 v(x_0,\nu) \nonumber \\
   &\leq& \alpha^2 v(x_0,\xi)+\beta^2 v(x_0,\nu).
\end{eqnarray}
Hence
\begin{equation}\label{}
  v(x_0,\xi)\leq v(x_0,\nu).
\end{equation}
Then we get the estimate,
\begin{equation}\label{}
  u_{\xi\xi}(x_0)\leq C_0(||u||_{C^1},||\varphi||_{C^3},
  ||\partial\Omega||_{C^4},||f||_{C^2},\min f,\kappa)
  (1+|u_{\nu\nu}(x_0)|),
\end{equation}
so that this case is also reduced to the purely normal case.
\end{proof}

\subsection{Second Normal Derivative Bounds On The Boundary}
In this section, we consider the double normal derivative estimate which is the most difficulty part in the Neumann problem for Hessian equations. Note we do not know boundary double tangential bound apriori, or it is hard to get this estimate due to the Neumann boundary condition in general. Compare this with Dirichlet problem in (\cite{CNS1}, \cite{CNS2}, \cite{T1}).\\
We give some definitions first.
Let
\begin{equation}\label{}
  h(x) = -d(x) +  d^2(x).
\end{equation}
We know from the classic book \cite{GT} section 14.6 that $h$ is $C^4$ in $\Omega_{\mu}$
for some constant $\mu \leq \widetilde{\mu}$ small depending on $\Omega$.
In terms of a principal coordinate system, see \cite{GT}  section 14.6, we have
\begin{equation}\label{Dd}
  [-D^2d(x_0)] = diag[\frac{\kappa_1(y_0)}{1-\kappa_1(y_0) d(x_0)}, \cdots , \frac{\kappa_{n-1}(y_0)}{1- \kappa_{n-1}(y_0)d(x_0)}, 0],
\end{equation}
and
\begin{equation}\label{}
  -Dd(x_0) = \nu(y_0).
\end{equation}
So $h$ also satisfied the following properties in $\Omega_{\mu}$:
\begin{equation}\label{h3}
   -\mu + \mu^2 \leq h \leq 0,
\end{equation}
\begin{equation}\label{h1}
  2 \geq |D h| \geq \frac{1}{2},
\end{equation}
\begin{equation}\label{h1}
  \frac{1}{3}k_1 \delta_{ij} \geq D^2 h \geq 2  k_0 \delta_{ij},
\end{equation}
\begin{equation}\label{h1}
   F^{ij} h_{ij} \geq k_0 (\mathcal{F}+1),
\end{equation}
provided $\mu \leq \widetilde{\mu}$ small depend on $||\partial \Omega||_{C^2}$. Here $k_1$ and $k_0$ are positive constants depend on $\kappa : = (\kappa_1, \cdots, \kappa_n)$.
It is easy to see
\begin{equation}\label{}
  \frac{Dh}{|Dh|} = \nu,
\end{equation}
for unit outer normal $ \nu$ on the boundary.\\
In order to do this estimate we construct barrier functions of $u_{\nu}$ on the boundary. Motivated by \cite{LTU}, \cite{Tru}, \cite{ITW} and \cite{Urbas}, we introduce the following functions.
In $\overline{\Omega}_{\mu}$,  we denote
\begin{equation}\label{}
  g(x):=1-\beta h,
\end{equation}
\begin{equation}\label{}
  G(x):= (A+\sigma M)h(x),
\end{equation}
\begin{equation}\label{}
  \psi(x) := |Dh|(x)\varphi(x,u).
\end{equation}
where $\sigma$, $\beta$, $\mu$, $A$ are positive constants to be chosen later.\\
Now we consider the sub barrier function,
\begin{equation}\label{PPP}
  P(x):=g(x)(Du \cdot Dh(x) - \psi(x)) - G(x).
\end{equation}
And we want to prove the following lemma.
\begin{lemma}\label{3.2}
Fix $\sigma = \frac{1}{2}$, for any $x \in \overline{\Omega}_{\mu}$ , if chosen $\beta$ large, $\mu$ small, $A$ large in proper sequence, we have
\begin{equation}\label{}
  P(x) \geq 0.
\end{equation}
\end{lemma}
\begin{proof}
We use maximum principle to prove this lemma.
First we assume the function attains its minimum point $x_0$ in the interior of $\Omega_{\mu}$.
We derivative this function twice,
\begin{equation}\label{p1}
  P_i = g_i(\sum\limits_l u_l h_l - \psi) + g(\sum\limits_l u_{li}h_l +\sum\limits_l u_l h_{li}-\psi_i)-G_i,
\end{equation}
and
\begin{eqnarray}\label{p2}
  P_{ij} &=& g_{ij}(\sum\limits_l u_l h_l - \psi) + g_i(\sum\limits_l u_{lj}h_l + \sum\limits_l u_l h_{lj}-\psi_j) \nonumber\\
  & & +g_j(\sum\limits_l u_{li}h_l + \sum\limits_l u_l h_{li}-\psi_i)-G_{ij}\\
  & &+g(\sum\limits_l u_{lij}h_l + \sum\limits_l u_{li}h_{lj} + \sum\limits_l u_{lj}h_{li}+\sum\limits_l u_l h_{lij}-\psi_{ij}). \nonumber
\end{eqnarray}
At the minimum point $x_0$, as before we can assume $(u_{ij}(x_0))$ is diagonal. Contracting \eqref{p2} with $F^{ij}$, we get
\begin{eqnarray}
  F^{ij} P_{ij} &=& F^{ij}g_{ij}(\sum\limits_l u_l h_l - \psi)+ 2 g_i F^{ij} (\sum\limits_{l} u_{lj} h_l + \sum\limits_l u_l h_{lj}-\psi_j)  \nonumber\\
   & & +gF^{ij}(\sum\limits_l u_{lij} h_l + 2 \sum\limits_l u_{li} h_{lj}+ \sum\limits_l u_l h_{lij}-\psi_{ij})  \nonumber\\
   & &-F^{ij}G_{ij} \nonumber\\
   & \leq &   \beta C_3(||u||_{C^1},||\partial \Omega||_{C^3},||\varphi||_{C^2},||f||_{C^1}) (\mathcal{F} + 1)\\\label{p1a}
   & &- (A+\sigma M)k_0 (\mathcal{F}+1) - 2 \beta F^{ii} u_{ii} h_i^2 + 2 F^{ii}u_{ii} h_{ii}g. \nonumber
\end{eqnarray}
Where in the second inequality we use
\begin{equation}\label{eps}
  |\beta h| \leq \beta \frac{\mu}{2} \leq \frac{1}{2},
\end{equation}
which in turn implies that
\begin{equation}\label{g}
1 \leq g \leq \frac{3}{2}.
\end{equation}
We choose $\mu \leq \frac{1}{\beta}$ in \eqref{eps}.\\
Then we divided the index $1\le i \le n$ into two categories.\\
(\textrm{i}) If
\begin{equation}\label{}
  |\beta h_i^2| \leq \frac{k_0}{2},
\end{equation}
we say $i \in \bf B$.\\
We choose $\beta \geq 2 n k_0$, in order to let
\begin{equation}\label{hi}
|h_i^2|\leq \frac{1}{4n}.
\end{equation}

\noindent(\textrm{ii}) If
\begin{equation}\label{}
  |\beta h_i^2| \geq \frac{k_0}{2},
\end{equation}
we denote $i \in \bf G$.\\
For any $i \in \bf G$, we use $P_i(x_0)=0$ to get
\begin{equation}\label{}
  u_{ii} = \frac{A+\sigma M}{g} + \frac{ \beta(\sum\limits_l u_l h_l -\psi)}{g} - \frac{\sum\limits_l u_l h_{li}}{h_i} + \frac{\psi_i}{h_i}.
\end{equation}
Because $|h_i|^2> \frac{k_0}{2\beta}$ and \eqref{g}, we have that
\begin{equation}\label{s}
  \mid \frac{\beta(\sum\limits_l u_l h_l - \psi)}{g}-\frac{\sum\limits_l u_l h_{li}}{h_i}+\frac{\psi_i}{h_i}\mid \leq \beta C_4(k_0 ,||u||_{C^1},||\partial \Omega||_{C^2},||\varphi||_{C^1}).
\end{equation}
By chosen $A$ large such that $\frac{A}{3} \geq \beta C_4$, we infer
\begin{equation}\label{uG}
  \frac{4A}{3} + \sigma M \geq u_{ii} \geq \frac{A}{3} + \frac{2 \sigma M}{3}, \indent for \indent i \in \bf G.
\end{equation}
Due to $2 \geq |D h| \geq \frac{1}{2}$ and \eqref{hi}, there is a $i_0 \in \bf G$, say $i_0 = 1$, such that
\begin{equation}\label{}
  h_{1}^2 \geq \frac{1}{4 n}.
\end{equation}
Then we continue to compute the equation of $P$,
\begin{eqnarray}\label{132}
  F^{ij}P_{ij} &\leq& [\beta C_3-(A+\sigma M) k_0] (\mathcal{F}+1) \nonumber\\
   & & - 2 \beta \sum\limits_{i\in \bf G}F^{ii}u_{ii}h_i^2 - 2 \beta \sum\limits_{i\in \bf B}F^{ii}u_{ii} h_i^2\nonumber \\
   & & + k_1 \sum\limits_{u_{ii} \geq 0} F^{ii}u_{ii} + 4 k_0 \sum\limits_{u_{ii}<0}F^{ii} u_{ii}.
   \end{eqnarray}
Since
\begin{equation}\label{132a}
- 2 \beta \sum\limits_{i\in \bf G}F^{ii}u_{ii}h_i^2 \le - 2 \beta F^{11}u_{11}h_1^2 \le - \frac{\beta}{2 n} F^{11} u_{11},
\end{equation}
and
\begin{equation}\label{132b}
- 2 \beta \sum\limits_{i\in \bf B}F^{ii}u_{ii} h_i^2 \le - 2 \beta \sum\limits_{i\in \bf {B}, u_{ii} < 0}F^{ii}u_{ii} h_i^2 \le -k_0\sum\limits_{ u_{ii} < 0}F^{ii}u_{ii},
\end{equation}
it follows that
\begin{equation}\label{132c}
- 2 \beta \sum\limits_{i\in \bf G}F^{ii}u_{ii}h_i^2 - 2 \beta \sum\limits_{i\in \bf B}F^{ii}u_{ii} h_i^2 + 4 k_0 \sum\limits_{u_{ii}<0}F^{ii} u_{ii} \le - \frac{\beta}{2 n} F^{11} u_{11}.
\end{equation}
From \eqref{132} and \eqref{132c}, we have
\begin{eqnarray}\label{78}
  F^{ij}P_{ij}
   &\leq& [\beta C_3-(A+\sigma M)k_0] (\mathcal{F}+1) \label{77}\nonumber\\
   & & - \frac{\beta}{2 n} F^{11} u_{11}+  k_1 \sum\limits_{u_{ii} \geq 0} F^{ii}u_{ii}.
\end{eqnarray}

Now we analysis the above terms case by case. Without generality, we assume that $u_{22} \geq \cdots \geq u_{nn}$.\\
{\bf{Case 1:}} $u_{ii} \geq 0$, for all $i$.\\
  This is the most easy case. Using equation, we get
  \begin{equation}\label{}
   kf = \sum\limits_{u_{ii} \geq 0} F^{ii} u_{ii}.
  \end{equation}
  If we choose $A > \frac{(C_3 \beta+k_1 k\max f)}{k_0} $, then from  \eqref{78} we have
  \begin{equation}\label{81}
    F^{ij}P_{ij} < 0.
  \end{equation}
  In the following cases we can assume $u_{nn} <0$.\\
{\bf{Case 2:}} $\frac{k_0}{2 k_1} u_{11} \geq |u_{nn}|$.\\
  Due to equation, we have
  \begin{equation}\label{eq}
   kf = \sum\limits_{u_{ii} \geq 0} F^{ii}u_{ii} + \sum\limits_{u_{ii} < 0} F^{ii}u_{ii}.
  \end{equation}
  The terms in line \eqref{78} become
  \begin{eqnarray}\label{141}
    - \frac{\beta}{2 n} F^{11} u_{11}+  k_1 \sum\limits_{u_{ii} \geq 0} F^{ii}u_{ii} &\leq& k_1 (kf - \sum\limits_{u_{ii}<0} F^{ii} u_{ii})\nonumber \\
     &\leq & k_1 k f - k_1 \mathcal{F} u_{nn}\nonumber\\
     & \leq & k_1 k f +  \frac{k_0}{2} \mathcal{F} u_{11}\nonumber\\
     & \leq & k_1 k f +  k_0 \mathcal{F}  (\frac{2A}{3} + \frac{\sigma M}{2} ).
  \end{eqnarray}
  Using \eqref{141}, and choose $A > \frac{3(C_3 \beta+k_1k\max f)}{k_0} $ in \eqref{78}, then we obtain the result \eqref{81}.\\

  In the following cases we  assume $$u_{nn} <0, \quad |u_{nn}|\geq \frac{k_0}{2 k_1} u_{11}.$$ We denote $\lambda :=(u_{11},\cdots,u_{nn})$ and choose $A \geq 2\sigma$.\\
 {\bf{Case 3:}}  $\sigma_{k-1}(\lambda | 1) \geq \delta_1 (-u_{nn})\sigma_{k-2}(\lambda | 1 n)$,  for small positive constant $\delta_1$ chosen in later case.\\
  If $u_{11} \geq u_{22}$, we know from \eqref{inqu12} that,
  \begin{equation}\label{106}
    u_{11} \sigma_{k-2} (\lambda | 1n) \geq \frac{k-1}{n-1} \sigma_{k-1}(\lambda | n).
  \end{equation}
  Otherwise $u_{11} \leq u_{22}$, we have from \eqref{uG}, \eqref{F1}  and  \eqref{inqu12}  that
  \begin{eqnarray}
    u_{11} \sigma_{k-2} (\lambda | 1n) &\geq& (\frac{A}{3} + \frac{2 \sigma M}{3})\sigma_{k-2} (\lambda | 2n) \nonumber\\
     &\geq& \frac{2\sigma }{3C_0}u_{22}\sigma_{k-2} (\lambda | 2n) \nonumber\\
     &\geq& \frac{k-1}{n-1} \frac{2\sigma}{3C_0}\sigma_{k-1}(\lambda | n).\label{109}
  \end{eqnarray}
  We infer from the hypothesis
  \begin{eqnarray}
    F^{11} &=& \sigma_{k-1}(\lambda | 1) \nonumber\\
     &\geq& \delta_1 (-u_{nn}) \sigma_{k-2}(\lambda|1n)\nonumber\\
     &\geq&  \delta_1 \frac{k_0}{2  k_1} u_{11} \sigma_{k-2}(\lambda|1n).\label{112}
  \end{eqnarray}
  Note we only use hypothesis of Case 3 in the first inequality above.\\

Using \eqref{in5}, and assumption $u_{nn} < 0$, we have from \eqref{8} that
\begin{equation}\label{inq13}
\frac{1}{n-k+1} \mathcal{F} \leq F^{nn}.
\end{equation}
  Assuming $C_0 \geq 1$ such that $\sigma = \frac{1}{2} \leq \frac{3C_0}{2}$, then we substitute \eqref{106} and \eqref{109} into \eqref{112}, and using \eqref{inq13},
  \begin{eqnarray}
   F^{11}&\ge &  \delta_1 \frac{k_0}{2  k_1} u_{11} \sigma_{k-2}(\lambda|1n)\nonumber\\
   &\geq&  \delta_1 \frac{k-1}{n-1} \frac{k_0}{2 k_1} \frac{2\sigma}{3C_0}\sigma_{k-1}(\lambda | n) \label{92}\nonumber\\
     & \geq&  \frac{k-1}{(n-1)(n-k+1)}\frac{ k_0\delta_1 \sigma}{3k_1C_0}\mathcal{F}.
  \end{eqnarray}
 Using \eqref{F1}, and  we choose $\beta \geq \frac{9 n(n-k+1)(n-1) k_1^2 C_0^2}{(k-1)k_0 \delta_1 \sigma^2}$, such that for the last two terms in \eqref{78} we have
  \begin{eqnarray}\label{147}
   &&- \frac{\beta}{2 n} F^{11} u_{11}+  k_1 \sum\limits_{u_{ii} \geq 0} F^{ii}u_{ii}\nonumber\\&\leq& [-\frac{(k-1)\beta k_0 \delta_1 \sigma}{6 n(n-1)(n-k+1) k_1 C_0}(\frac{A}{3}+\frac{2\sigma M}{3})\nonumber\\
  & & + k_1 C_0 (M+1)] \mathcal{F}\nonumber\\
   &\leq& [-(\frac{(k-1)\beta k_0 \delta_1 \sigma^2}{9n(n-k+1)(n-1) k_1 C_0}- k_1C_0)M\nonumber\\
   & &+(-\frac{ A \beta k_0(k-1)\delta_1 \sigma}{18n(n-k+1)(n-1)k_1 C_0}+k_1 C_0)]\mathcal{F} \nonumber\\
   &\leq& 0.
  \end{eqnarray}
  So choose $A > \frac{(C_3 \beta+k_1 k\max f)}{k_0}+2\sigma $  in \eqref{78}, and using \eqref{147}, we obtain the inequality \eqref{81}.\\
{\bf{Case 4:}} $0 \leq \sigma_{k-1}(\lambda |1) \leq \delta_1 (-u_{nn})\sigma_{k-2}(\lambda|1n)$.\\
  By hypothesis and for $i \ge 2,$
   \begin{eqnarray}\label{148}
   \sigma_{k-1}(\lambda | 1)-u_{ii}\sigma_{k-2}(\lambda | 1i)= \sigma_{k-1}(\lambda | 1i).
  \end{eqnarray}
   We compute as follows,
  \begin{eqnarray}
    k \sigma_k(\lambda | 1) &=& \sum\limits_{i=2}^n u_{ii}\sigma_{k-1}(\lambda | 1i)  \nonumber\\
     &\leq&\sum\limits_{u_{ii}\geq 0, i\neq1} u_{ii}[\delta_1(-u_{nn}\sigma_{k-2}(\lambda|1n))
     -u_{ii}\sigma_{k-2}(\lambda|1i)]\nonumber\\
     & &+\sum\limits_{u_{ii}<0,i\neq1} u_{ii}(-u_{ii}\sigma_{k-2}(\lambda|1i)) \nonumber\\
     &\leq& -u_{nn}\sum\limits_{u_{ii}\geq0,i\neq1}\delta_1 u_{ii}\sigma_{k-2}(\lambda|1n) - u_{nn}^2\sigma_{k-2}(\lambda|1n) \nonumber\\
     &\leq& -n \delta_1 u_{nn} u_{22} \sigma_{k-2}(\lambda|1n) - u_{nn}^2 \sigma_{k-2}(\lambda|1n).
  \end{eqnarray}
  Using \eqref{F1} and \eqref{uG}, we continue
  \begin{eqnarray}
    k\sigma_{k}(\lambda|1) &\leq& -n\delta_1 C_0(M+1) u_{nn}\sigma_{k-2}(\lambda|1n)-u_{nn}^2\sigma_{k-2}
    (\lambda|1n)  \nonumber\\
     &\leq&  -n\delta_1 C_0 \frac{3}{2\sigma}u_{11}u_{nn} \sigma_{k-2}(\lambda|1n) - u_{nn}^2\sigma_{k-2}(\lambda|1n)\nonumber\\
     &\leq& \frac{3k_1 n\delta_1 C_0}{k_0 \sigma} u_{nn}^2 \sigma_{k-2}(\lambda|1n) - u_{nn}^2 \sigma_{k-2}(\lambda|1n).
  \end{eqnarray}
  Now we let $\delta_1 = \frac{k_0 \sigma}{6k_1 n C_0}$. As in \eqref{112} and \eqref{92}, we obtain
  \begin{eqnarray}
    k\sigma_{k}(\lambda|1) &\leq& -\frac{u_{nn}^2}{2}\sigma_{k-2}(\lambda|1n)\nonumber \\
     &\leq&  u_{nn} \frac{k-1}{(n-k+1)(n-1)}\frac{\sigma k_0 }{6k_1 C_0}\mathcal{F}\nonumber\\
     &\leq& -\frac{k-1}{(n-k+1)(n-1)} \frac{\sigma k_0^2 }{12k_1^2 C_0} u_{11}\mathcal{F}.
  \end{eqnarray}
  Inserting \eqref{uG} into above inequality, we have
  \begin{eqnarray}
    -\frac{\beta}{2n} F^{11}u_{11} &=& -\frac{\beta}{2n} (f - \sigma_k (\lambda |1)) \\
     &\leq&  -\frac{\beta}{2n} f - \frac{\beta \sigma k_0^2} {24  k n (n-k+1) k_1^2 C_0}\frac{k-1}{n-1} (\frac{A}{3}+\frac{2\sigma M}{3})\mathcal{F}. \nonumber
  \end{eqnarray}
  If we choose $\beta \geq \frac{36 k n(n-k+1)(n-1) k_1^3 C_0^2}{(k-1)\sigma^2 k_0^2 }$ , such that for the last two terms in \eqref{78} we get
  \begin{equation}\label{154}
    - \frac{\beta}{2 n} F^{11} u_{11}+  k_1 \sum\limits_{u_{ii} \geq 0} F^{ii}u_{ii} \leq -\frac{\beta}{2n} f < 0.
  \end{equation}
  Finally, we choose $A  \geq \frac{3(C_3 \beta +k_1 k\max f)}{k_0}$ in \eqref{78}, and using \eqref{154}, we obtain the inequality \eqref{81} which contradicts with $0 \leq F^{ij}P_{ij}$ at minimum point $x_0$.

Then the function $P$ attains its minimum on the boundary of $\Omega_{\mu}$.\\
 Now we treat the boundary value of $P$.  On $\partial \Omega$, it is easily to see
\begin{equation}\label{}
  P = 0.
\end{equation}
On the $\partial \Omega_{\mu} / \partial \Omega$, we have
\begin{equation}\label{}
  P \geq - C_5(k, \max f, ||u||_{C^1}, ||\varphi||_{C^0})+(A+\sigma M)\frac{\mu}{2} \geq 0,
\end{equation}
provided $A \geq \frac{2 C_5}{\mu}$.\\
We conclude that we first choose $\delta_1 = \frac{k_0 \sigma}{6k_1 n C_0}$, then $\beta= \frac{36k n(n-k+1)(n-1)k_1^3 C_0^2}{(k-1)\sigma^2 k_0^2} + \frac{9n(n-k+1)(n-1)k_1^2 C_0^2}{(k-1)k_0 \delta_1 \sigma^2}+2nk_0$ , then $\mu=\min \{ \mu_0, \frac{1}{\beta} \}$, finally $A= \frac{3(C_3\beta +k_1 k \max f)}{k_0} +3\beta C_4+ 2\sigma +1+\frac{2C_5}{\mu}$. Using  the maximal principle for the function $P(x)$,  we get
\begin{equation*}\label{}
  P(x) \geq 0, \quad \text {in }\quad \Omega.
\end{equation*}
\end{proof}
Similarly, we can also find super barrier function of $u_{\nu}$.
\begin{lemma}
 Let $\overline{P} := g(x)(Du \cdot Dh(x) - \psi(x))+ G(x)$. Fix $\sigma =\frac{1}{2} $, for any $x \in \overline{\Omega}_{\mu}$ , if chosen $\beta$ large, $\mu$ small, $A$ large in proper sequence, we have
\begin{equation}\label{}
  \overline{P}(x) \leq 0.
\end{equation}
\end{lemma}
\begin{proof}
We assume the function attains its maximum point $x_0$ in the interior of $\Omega_{\mu}$.
We derivative this function twice,
\begin{equation}\label{pp1}
  \overline{P}_i = g_i(\sum\limits_l u_l h_l - \psi) + g(\sum\limits_l u_{li}h_l +\sum\limits_l u_l h_{li}-\psi_i)+G_i,
\end{equation}
and
\begin{eqnarray}\label{pp2}
  \overline{P}_{ij} &=& g_{ij}(\sum\limits_l u_l h_l - \psi) + g_i(\sum\limits_l u_{lj}h_l + \sum\limits_l u_l h_{lj}-\psi_j)\nonumber\\
  & & +g_j(\sum\limits_l u_{li}h_l + \sum\limits_l u_l h_{li}-\psi_i)+G_{ij}\nonumber\\
  & &+g(\sum\limits_l u_{lij}h_l + \sum\limits_l u_{li}h_{lj} + \sum\limits_l u_{lj}h_{li}+\sum\limits_l u_l h_{lij}-\psi_{ij}).
\end{eqnarray}
At the maximum point $x_0$, as before we can assume $(u_{ij}(x_0))$ is diagonal. Contracting \eqref{pp2} with $F^{ij}$, we get
\begin{eqnarray}
  F^{ij} \overline{P}_{ij} &=& F^{ij}g_{ij}(\sum\limits_l u_l h_l - \psi)+ 2 g_i F^{ij} (\sum\limits_{l} u_{lj} h_l + \sum\limits_l u_l h_{lj}-\psi_j) \nonumber\\
   & & +gF^{ij}(\sum\limits_l u_{lij} h_l + 2 \sum\limits_l u_{li} h_{lj}+ \sum\limits_l u_l h_{lij}-\psi_{ij}) +F^{ij}G_{ij}\nonumber\\
   & \geq &   -\beta C_6(||u||_{C^1},||\partial \Omega||_{C^3},||\varphi||_{C^2},||f||_{C^1}) (\mathcal{F} + 1)\nonumber\\
   & &+ (A+\sigma M)k_0 (\mathcal{F}+1) - 2 \beta F^{ii} u_{ii} h_i^2 + 2 F^{ii}u_{ii} h_{ii}g.
\end{eqnarray}
As before we divided the index $1\le i \le n $ into two categories.\\
(\textrm{i}) If
\begin{equation}\label{}
  |\beta h_i^2| \leq \frac{k_0}{2},
\end{equation}
we say $i \in \bf B$.\\
We choose $\beta \geq 2 n k_0$, in order to let
\begin{equation}\label{phi}
|h_i^2|\leq \frac{1}{4n}.
\end{equation}

\noindent(\textrm{ii}) If
\begin{equation}\label{}
  |\beta h_i^2| \geq \frac{k_0}{2},
\end{equation}
we denote $i \in \bf G$.\\
For any $i \in \bf G$, we use $\overline{P}_i(x_0)=0$ to get
\begin{equation}\label{167}
  u_{ii} = -\frac{A+\sigma M}{g} + \frac{ \beta(\sum\limits_l u_l h_l -\psi)}{g} - \frac{\sum\limits_l u_l h_{li}}{h_i} + \frac{\psi_i}{h_i}.
\end{equation}
Because $|h_i|^2> \frac{k_0}{2\beta}$ and \eqref{167}, we have that
\begin{equation}\label{ps}
  \mid \frac{\beta(\sum\limits_l u_l h_l - \psi)}{g}-\frac{\sum\limits_l u_l h_{li}}{h_i}+\frac{\psi_i}{h_i}\mid \leq \beta C_4(k_0 ,||u||_{C^1},||\partial \Omega||_{C^2},||\varphi||_{C^1}).
\end{equation}
By chosen $A$ large such that $\frac{A}{3} \geq \beta C_4$, we infer
\begin{equation}\label{puG2}
  - \frac{4A}{3} - \sigma M \leq u_{ii} \leq -\frac{A}{3} - \frac{2 \sigma M}{3}, \indent for \indent i \in \bf G.
\end{equation}
Due to $2 \geq |D h| \geq \frac{1}{2}$ and \eqref{phi}, there is a $i_0 \in \bf G$, say $i_0 = 1$, such that
\begin{equation}\label{}
  h_{1}^2 \geq \frac{1}{4 n}.
\end{equation}
Then we continue to compute the equation of $\overline{P}$,
\begin{eqnarray}
  F^{ij}\overline{P}_{ij} &\geq& [-\beta C_6+(A+\sigma M) k_0] (\mathcal{F}+1)  \nonumber\\
   & & - 2 \beta \sum\limits_{i\in \bf G}F^{ii}u_{ii}h_i^2 - 2 \beta \sum\limits_{i\in \bf B}F^{ii}u_{ii} h_i^2  \nonumber\\
   & & + 4k_0 \sum\limits_{u_{ii} \geq 0} F^{ii}u_{ii} + k_1 \sum\limits_{u_{ii}<0}F^{ii} u_{ii}.  \label{p78}
\end{eqnarray}
We treat some terms in last formula, first
\begin{equation}\label{171a}
  - 2 \beta \sum\limits_{i\in \bf G}F^{ii}u_{ii}h_i^2 \ge - 2 \beta F^{11}u_{11}h_1^2 \ge - \frac{\beta}{2 n} F^{11} u_{11},
\end{equation}
then
\begin{eqnarray}\label{171b}
&&- 2 \beta \sum\limits_{i\in \bf B}F^{ii}u_{ii} h_i^2 \ge - 2 \beta \sum\limits_{i\in \bf B, u_{ii} \ge 0}F^{ii}u_{ii} h_i^2 \nonumber\\
 && \ge -k_0 \sum\limits_{i\in \bf B, u_{ii} \ge 0}F^{ii}u_{ii}= -k_0 \sum\limits_{ u_{ii} \ge 0}F^{ii}u_{ii} .
\end{eqnarray}
It follows that
\begin{eqnarray}
   &&- 2 \beta \sum\limits_{i\in G}F^{ii}u_{ii}h_i^2 - 2 \beta \sum\limits_{i\in B}F^{ii}u_{ii} h_i^2 \nonumber\\
   && + 4k_0 \sum\limits_{u_{ii} \geq 0} F^{ii}u_{ii} + k_1 \sum\limits_{u_{ii}<0}F^{ii} u_{ii} \nonumber\\
   &&\geq  - \frac{\beta}{2 n} F^{11} u_{11} +  k_1 \sum\limits_{u_{ii} < 0} F^{ii}u_{ii}. \label{171c}
\end{eqnarray}
Then we have
\begin{eqnarray}
  F^{ij}\overline{P}_{ij} &\geq&   [-\beta C_6+(A+\sigma M)k_0] (\mathcal{F}+1) \nonumber\\
   & &  - \frac{\beta}{2 n} F^{11} u_{11}+  k_1 \sum\limits_{u_{ii} < 0} F^{ii}u_{ii}. \label{p78a}
\end{eqnarray}

This is easy when $u_{11} < 0$, because we have by \eqref{inqu14} and \eqref{inqu13}
\begin{equation}\label{}
  F^{11} \geq c(k,n) \mathcal{F}.
\end{equation}
From \eqref{F1} and \eqref{puG2} we obtain
\begin{equation}\label{pp78}
 - \frac{\beta}{2 n} F^{11} u_{11}+  k_1 \sum\limits_{u_{ii} < 0} F^{ii}u_{ii} \geq \frac{\beta c}{2n}\mathcal{F}(\frac{A}{3}+\frac{2\sigma M}{3})-k_1 \mathcal{F} C_0 (1+M).
\end{equation}
If we choose $\beta \geq \frac{3 n k_1 C_0}{c \sigma}$ and $A\geq 2\sigma + \frac{\beta C_6}{k_0}$ , then by \eqref{p78a} and \eqref{pp78} we get
\begin{equation}\label{}
  F^{ij}\overline{P}_{ij} >0.
\end{equation}

Then the function $\overline{P}$ attains its maximum on the boundary of $\Omega_{\mu}$.\\
On $\partial \Omega$, it is easily to see
\begin{equation*}\label{}
  \overline{P} = 0.
\end{equation*}
On the $\partial \Omega_{\mu} / \partial \Omega$, we have
\begin{equation}\label{}
  \overline{P} \le  C_7(k, \max f, ||u||_{C^1}, ||\varphi||_{C^0})-(A+\sigma M)\frac{\mu}{2} \le 0,
\end{equation}
provided $A \geq \frac{2 C_7}{\mu}$.\\
We conclude that we first choose  $\beta \geq \frac{3 n k_1 C_0}{c \sigma}$ , then $\mu=\min \{ \mu_0, \frac{1}{\beta} \}$, finally $A\geq 2\sigma + \frac{\beta C_6}{k_0} +3\beta C_4+ 1 +\frac{2C_7}{\mu}$. Using the maximal principle for the function $\overline{P}(x)$,  we get
\begin{equation*}\label{}
  \overline{P}(x) \le 0, \quad \text {in }\quad \Omega.
\end{equation*}
\end{proof}

Using the barrier functions, we have the main normal-normal second derivative estimate in this section.
\begin{theorem}\label{nn}
  Let $\Omega$ be a bounded $C^{4}$ uniformly convex domain in $\mathbb{R}^n$, $\nu$ is the outer unit normal vector of $\p \O$. If $u\in C^4(\Omega)\cap C^3(\overline{\Omega})$ a $k$- admissible solution of Neumann problem \eqref{NP}. Where $f\in C^{2}(\overline{\Omega}\times\mathbb{R})$ is positive and $\varphi \in C^{3}(\overline{\Omega}\times\mathbb{R})$ is non-increasing in $z$. Then we have
  \begin{equation}\label{}
    \sup\limits_{\partial \Omega} |u_{\nu\nu}| \leq C,
  \end{equation}
where constant $C$ depends on $n$, $k$, $||u||_{C^1}$, $\min f$, $||\varphi||_{C^3}$, $||f||_{C^2}$, convexity of $\partial \Omega$ and $||\partial \Omega||_{C^4}$.
\end{theorem}
\begin{proof}
Assume $z_0$ is the maximum point of $u_{\nu\nu}$, we have
\begin{eqnarray}
  0 & \geq & P_{\nu}(z_0) \\
   &\geq & g (\sum\limits_l u_{l\nu} h_l + u_l h_{l\nu} - \psi_{\nu})-(A+\sigma M)h_{\nu}\\
   &\geq & u_{\nu\nu}-C(||u||_{C^1},||\partial \Omega||_{C^2},
   ||\psi||_{C^2})- (A+\sigma M)
\end{eqnarray}
In the second inequality we assume  $u_{\nu\nu}(z_0)\geq 0$. Then we get
\begin{equation}\label{}
  \sup\limits_{\partial \Omega} u_{\nu\nu} \leq  C+\sigma M.
\end{equation}
Similarly, by $0 \leq \overline{P}_{\nu}(z_0)$ here $z_0$ is the minimum point of $u_{\nu\nu}$, we get
\begin{equation}\label{}
   \inf\limits_{\partial \Omega} u_{\nu\nu} \geq  -C -\sigma M .
\end{equation}
So chosen $\sigma=\frac{1}{2}$ as in the previous lemmas, we get the estimate
\begin{equation}\label{}
    \sup\limits_{\partial \Omega} |u_{\nu\nu}| \leq C.
  \end{equation}
\end{proof}

 {\bf Proof of The Theorem~\ref{C^2}}: Combining Lemma~\ref{tn}, Lemma~\ref{gl} and the Theorem~\ref{nn}, we complete the proof of the Theorem~\ref{C^2}.

\section{Existence of the boundary problem}
In this section we complete the proof of the Theorem~\ref{main}. As in  \cite{LTU}, by combining Theorems \ref{C^0}, \ref{MQX} and \ref{C^2} with the global second derivative H\"{o}lder estimates (see \cite{LionT} or \cite{LT86}), we get a global estimate
\begin{equation}\label{}
  ||u||_{C^{2,\alpha}(\overline{\Omega})} \leq C
\end{equation}
for $k$- admissible solution, where $C$, $\alpha$ depending on $k$, $n$, $\Omega$, $||\Omega||_{C^4}$, $||f||_{C^2}$, $\min f$ and $||\varphi||_{C^3}$. Applying the method of continuity (see \cite{GT}, Theorem 17.28), we complete the proof of Theorem~ \ref{main}.

\end{document}